\documentclass[10pt,reqno]{article}

\usepackage{amsmath}
\usepackage{amsthm}
\usepackage{amssymb}
\usepackage{array}

\theoremstyle{plain}
\newtheorem{theorem}{Theorem}

\newtheorem{prop}{Proposition}

\newtheorem{ass}{Assumption}
\newtheorem{alg}{Algorithm}

\theoremstyle{remark}

\newcommand{\beq}{\begin{equation}}
\newcommand{\eeq}{\end{equation}}
\newcommand{\beql}[1]{\begin{equation}\label{#1}}
\newcommand{\beal}[1]{\begin{equation}\label{#1}\begin{aligned}}
\newcommand{\eeal}{\end{aligned}\end{equation}}

\newcommand{\ben}{\begin{eqnarray}}
\newcommand{\een}{\end{eqnarray}}
\newcommand{\nn}{\nonumber}

\newcommand{\dd}{\mathrm{d}}

\newcommand{\D}{D}

\newcommand{\Beq}{\begin{equation} \left\{ \begin{array}{rcll}}
\newcommand{\Eeq}[1]{ \end{array} \right.\label{#1}\end{equation}}

\usepackage{color}
\usepackage[left=3cm,right=3cm,top=3.5cm,bottom=4cm,marginparwidth=1.7in]{geometry}
\usepackage{graphicx}
\usepackage{subfig}
\usepackage[toc,page]{appendix}

\usepackage{url}
\RequirePackage{ifthen}
\makeatletter
\newcommand{\logmessage}[1]{\@latex@warning{#1}}
\makeatother
\IfFileExists{Bibliographie/Users_Guide.txt}{
  }{
  \IfFileExists{../Bibliographie/Users_Guide.txt}{
    }{
    \IfFileExists{../../Bibliographie/Users_Guide.txt}{
      }{
        \IfFileExists{../../../Bibliographie/Users_Guide.txt}{
          }{
            \IfFileExists{../../../../Bibliographie/Users_Guide.txt}{
              }{
        \logmessage{Directory 'Bibliographie' not found}
          }}}}}

\begin{document}


\title{Determining anisotropic conductivity using diffusion tensor imaging data in magneto-acoustic tomography with magnetic induction \thanks{}}

\author{Habib Ammari\thanks{\footnotesize Department of Mathematics,
ETH Z\"urich,
R\"amistrasse 101, CH-8092 Z\"urich, Switzerland (habib.ammari@math.ethz.ch). } \and
Lingyun Qiu\thanks{\footnotesize Institute for Mathematics and its Applications, University of Minnesota, Minneapolis, MN 55455 (qiu.lingyun@ima.umn.edu). } \and
Fadil Santosa\thanks{\footnotesize Institute for Mathematics and its Applications, University of Minnesota, Minneapolis, MN 55455 (santosa@ima.umn.edu).}\and
Wenlong Zhang\thanks{\footnotesize Department of Mathematics and Applications,
Ecole Normale Sup\'erieure, 45, rue d'Ulm, 75230 Paris Cedex 05, France
( wenlong.zhang@ens.fr).}}

\date{}
\maketitle

\begin{abstract}
In this paper we present a mathematical and numerical framework for a procedure of imaging anisotropic electrical conductivity tensor by integrating magneto-acoutic tomography with data acquired from diffusion tensor imaging.   Magneto-acoustic Tomography with Magnetic Induction (MAT-MI) is a hybrid, non-invasive medical imaging technique to produce conductivity images with improved spatial resolution and accuracy.
Diffusion Tensor Imaging (DTI) is also a non-invasive technique for characterizing the diffusion properties of water molecules in tissues.  We propose a model for anisotropic conductivity in which the conductivity is proportional to the diffusion tensor.  Under this assumption, we propose an optimal control approach for reconstructing the anisotropic electrical conductivity tensor. We prove convergence and Lipschitz type stability of the algorithm and present numerical examples to illustrate its accuracy and feasibility.
\end{abstract}

\bigskip



\section{Introduction}
\numberwithin{equation}{section}

In this paper, we describe a method of reconstructing images of an anisotropic
conductivity tensor distribution inside an electrically conducting object using
Magneto-acoustic Tomography with Magnetic Induction (MAT-MI).
MAT-MI is a new noninvasive modality for imaging electrical conductivity  distributions of biological tissue \cite{He05,He14-2}. In the experiments, the biological tissue is placed in a static magnetic field. A pulsed magnetic field is applied to induce an eddy current inside the conductive tissue. In the process, the tissue emits ultrasound waves which can be measured around the tissue. The first step in the MAT-MI is to recover the acoustic source in the scalar wave equation from measurements at a set of locations around the object.
This problem has been studied in many works, including \cite{Finch09,Hal05,Hri08,Kuch08,Stefanov}. The second step in the MAT-MI is to reconstruct the electrical conductivity distribution from knowledge of the acoustic source.

Biological tissues are known to
have anisotropic conductivity values \cite{schwan, Tuch98}.  However, up to
now, all MAT-MI techniques have been devised to image isotropic conductivity distribution.
The fundamental questions arising in this case have been addressed in
\cite{Ammari15,Qiu2015}.  However, for the case of anisotropic conductivity, many basic questions remain.  For
instance, it is not known how many MAT-MI measurements are needed to uniquely determine the conductivity
tensor.

DTI is a non-invasive technique for characterizing the diffusion properties of water molecules in tissues (see e.g. \cite{Lebihan86} and the references therein). Imaging conductivity tensors in the tissue with DTI is based on the correlation property between diffusion and conductivity tensors \cite{Tuch98}. This linear relationship can be used to characterize the conductivity tensor. Once the conductivity directions of anisotropy are determined, one needs only to reconstruct a cross-property factor which is a scalar function. In \cite{Nachman13, woo}, it is shown how to recover this factor in Current Density Impedance Imaging. In \cite{timothee}, a multifrequency electrical impedance approach is developed for  estimating the ratio between the largest and the lowest eigenvalue of the electrical conductivity tensor.  An iterative procedure for reconstructing anisotropic conductivities from internal current densities has been proposed in \cite{seo}.

In the process of the MAT-MI experiment, the tissue is placed in a constant static magnetic
background field $B_0 = (0, 0, 1)$. A pulsed magnetic stimulation of the form $B_1u(t)$ is applied, where the vector field $B_1$ is constant and $u(t)$ is the time variation. Let $\gamma$ denote the conductivity, $E_{\gamma}$
 denote electric field, and $\Omega$ be the domain to be imaged. Then the electric field
satisfies the following Maxwell equations
\beq
\label{eqn:MAT-MI}
\left\{
\begin{array}{lcll}
\nabla\times E_{\gamma} &=& B_1 \quad&\mbox{in } \Omega, \\
\nabla\cdot (\gamma E_{\gamma}) &=& 0 \quad &\mbox{on } \Omega, \\
\gamma E_{\gamma} \cdot \nu &= & 0 \quad &\mbox{on } \partial \Omega.
\end{array} \right.
\eeq
The second step of MAT-MI is to reconstruct $\gamma$ from the known internal data $\nabla\cdot (\gamma E_{\gamma}\times B_0)$ on $\Omega$.

In this paper, we will consider the anisotropic conductivity $\gamma$. We assume that DTI has been performed on the tissue being imaged, that is, the diffusion tensor $D(x)$ has been found.  Next we follow \cite{Tuch98} and make
the assumption that conductivity is proportional to $D(x)$ and assume that the conductivity tensor $\gamma$ is of the form
\beq
\gamma(x)=\sigma(x) D(x) .
\eeq
The cross-property factor $\sigma$ is a scalar function to be reconstructed.  We will focus on the second step of MAT-MI combined with DTI, i.e., on reconstructing the cross-property factor $\sigma$ from the internal data given by $\nabla\cdot (\gamma E_{\gamma}\times B_0)$ with known conductivity tensor $D(x)$.

In the following, we assume that $D(x)$ is a positive definite symmetric matrix everywhere and write it as
$D=T^{'} \Sigma T$, where $D=\mathrm{diag}(e_1,e_2,e_3)$,
$e_1\geq e_2\geq e_3$ are the eigenvalues of $D(x)$. The columns of $T^{'}$ are the corresponding eigenvectors. As we can always write
$\sigma=\sigma_0e_1T^{'} \mathrm{diag}(1,e_2/e_1,e_3/e_1)T$,
we assume that $e_1=1$ hereinafter.

\section{Notation and preliminaries}
\label{sec:prelim}
In this section, we introduce the notation for the mathematical analysis.
Let $\Omega$ be a bounded Lipschitz domain in $\mathbb{R}^3$. A typical point $x=(x_1,x_2,x_3)\in\mathbb{R}^3$
denotes the spatial variable. Throughout this paper, the standard notation for H\"older and Sobolev spaces and their norms is used. If there is no confusion, we omit the dependence on the domain.
\begin{ass}\label{asmp:sigma}
   Let $\sigma$ and $D$ be positive functions belonging to $C^{1,\beta}$, $\beta>0$ and assume that
\begin{equation}
  c_1\leq \sigma(x) \leq c_2, \quad \forall x\in \Omega,
  \end{equation}
  and
\beq\label{ellipticity}
  \lambda\|\xi\|_2^2 \leq\xi^{'}D\xi\leq\|\xi\|_2^2,\quad \forall \xi\in R^3,
\eeq
for some constants $\lambda,c_1,c_2 >0$.
\end{ass}

Here we state several useful results on elliptic partial differential equations with Neumann boundary conditions.

\begin{def}\label{def:solution}
We say that $u\in H^1$ is a weak solution of the Neumann boundary value problem
  \begin{equation}\label{eqn:N-bd}
    \left\{
     \begin{array}{lll}
   \nabla \cdot (\sigma D\nabla u) & = - \nabla \cdot {E}, &\qquad \mbox{ in } \Omega,
  \\
     (\sigma D\nabla u + {E}) \cdot \nu & =  0 , &\qquad \mbox{ on } \partial\Omega,
  \end{array}
    \right.
  \end{equation}
  if
  \beq\label{def:solution}
  \int_{\Omega} \sigma D\nabla u \cdot \nabla \varphi \,\dd x = -\int_{\Omega} {E} \cdot \nabla \varphi \,\dd x, \quad \forall \varphi \in H^1.
  \eeq
\end{def}

\medskip\medskip

We give a brief proof of the following regularity result and standard energy estimate.

\begin{prop}\label{prop:reg}
Suppose that $\sigma$ and $D$ satisfy Assumption~\ref{asmp:sigma}.
For field ${E}\in L^2$, the Neumann problem \eqref{eqn:N-bd} has a solution $u\in H^1$.
The solution $u$ is unique up to an additive constant and satisfies the estimate
  \begin{equation}\label{grad-est}
    \|\nabla u\|_{L^2} \leq c_1^{-1}\lambda^{-1} \|{E}\|_{L^2}.
  \end{equation}
\end{prop}

\begin{proof}
The proof of the existence and uniqueness up to an additive constant is a standard result by the Lax-Milgram Theorem. We refer the readers to \cite{Taylor2011}. In the following, we prove the gradient estimate \eqref{grad-est}.

It follows from the ellipticity condition \eqref{ellipticity} that
  \[
    c_1\lambda \|\nabla u \|_{L^2}^2 \leq \int_{\Omega} \sigma \nabla u\cdot D\nabla u \, \dd x.
    \]
Taking the test function $\varphi$ in Definition~\ref{def:solution} to be the solution $u$, we have that
\[
\int_{\Omega} \sigma D\nabla u \cdot \nabla u \, \dd x =  - \int_{\Omega} {E} \cdot  \nabla u \, \dd x.
\]
Consequently, applying the Cauchy-Schwarz inequality, we obtain that
  \[
   c_1\lambda \|\nabla u \|_{L^2}^2 \leq \left| - \int_{\Omega} {E} \cdot  \nabla u \, \dd x \right| \leq \|\nabla u \|_{L^2} \|{E}\|_{L^2},
  \]
  and \eqref{grad-est} follows. 
\end{proof}

\begin{prop}
Let $\sigma$ and $D$ satisfy Assumption~\ref{asmp:sigma}. Then the system \eqref{eqn:MAT-MI} is uniquely solvable and there exists a constant $C$ and $C_i$ $(1\leq i\leq 3)$  depending on $\lambda$, $c_1$, $c_2$ and $\Omega$, such that
\beq
\label{EL2bd}
 \|{E}_{\sigma D} \|_{L^2} \leq C_1,
\eeq
 \beq
\label{inftyele}
 \|E_{\sigma D}\|_{L^{\infty}(\Omega)} \leq C_2,
\eeq
\beq
\label{C1}
 \|E_{\sigma D}\|_{C^{1,\gamma}(\Omega)} \leq C_3.
\eeq
Moreover, if $\sigma_1$ and $\sigma_2$ satisfy Assumption~\ref{asmp:sigma}, we have the following bound for the electric field difference,
\beq
\label{diffele}
 \|E_{\sigma_1D}-E_{\sigma_2D}\|_{L^2(\Omega)} \leq C \|\sigma_1-\sigma_2\|_{L^2(\Omega)}.
\eeq

\beq
\label{diffeleH1}
 \|E_{\sigma_1D}-E_{\sigma_2D}\|_{H^1(\Omega)} \leq C \|\sigma_1-\sigma_2\|_{H^1(\Omega)}.
\eeq
\end{prop}

\begin{proof}
Let us first reduce the system \eqref{eqn:MAT-MI} to a Neumann boundary value problem.
  Let $  {\tilde{E}} = \frac{1}{2}(-y,x,0)$.
   We can readily check that $\nabla \times {\tilde{E}} = {B}_1$. Hence $\nabla \times({E}_{\sigma D} - {\tilde{E}} ) = 0$ and we can write ${E}_{\sigma D}= {\tilde{E}} + \nabla u$. Substituting this into \eqref{eqn:MAT-MI}, we have that $u$ solves the
   Neumann boundary value problem
  \begin{equation}\label{eqn:E-decomp}
    \left\{
     \begin{array}{lll}
   \nabla \cdot (\sigma D\nabla u) & = - \nabla \cdot (\sigma D{\tilde{E}}), &\qquad \mbox{ in } \Omega,
  \\
    (\sigma D\nabla u + \sigma D{\tilde{E}})\cdot \nu & = 0, &\qquad \mbox{ on } \partial\Omega.
  \end{array}
    \right.
  \end{equation}
With the help of proposition \ref{prop:reg},  we get
 \[
 \|{E}_{\sigma D} \|_{L^2} \leq C_1.
 \]

 From the standard $L^p$ estimate for elliptic equations \cite[Chapter 9]{Gilbarg01} and the Sobolev Embedding Theorem, we know that ${E}_{\sigma D}$ is a bounded function in $W^{2,p}(\Omega)$ for any $p>2$. Hence,  ${E}_{\sigma D}$  is uniformly bounded by a constant
$C$, which depends only on $r_0$, $\lambda$, $c_1$, $c_2$, and $\Omega$. Then (\ref{inftyele}) is proved.

With the assumption of $C^{1,\gamma}$ property, we would obtain the $C^{2,\gamma}$ H\"older continuity\cite{Grisvard} of $u$, i.e., the $C^{1,\gamma}$ continuity of $E_{\sigma D}$. Estimate \eqref{C1} has been proven.

Next, we estimate the electric field difference. We denote $E_i=E_{\sigma_iD}$, for $i=1,2$. Note that ${E}_1 - {E}_2$ is curl-free. We set
  \[
  \nabla u = {E}_1 - {E}_2.
  \]
  Then, $u$ satisfies the equation
  \begin{equation}\label{eqn:diffele}
  \left\{
  \begin{array}{lll}
   \nabla \cdot (\sigma_1 D\nabla u) & =  -\nabla \cdot ((\sigma_1 - \sigma_2)D{E}_2), &\qquad \mbox{ in } \Omega,
  \\
     D\nabla u \cdot \nu & =  0 , &\qquad \mbox{ on } \partial\Omega.
  \end{array}
  \right.
\end{equation}
  With the same argument for proving (\ref{EL2bd}), we obtain that
  \[
  \|\nabla u\|_{L^2} \leq c_1^{-1}\lambda^{-1} \|(\sigma_1 -\sigma_2) D{E}_2\|_{L^2}.
  \]
Thus, we conclude from (\ref{inftyele}) that
  \[
  \|{E}_1 - {E}_2\|_{L^2} \leq C \|(\sigma_1 -\sigma_2)\|_{L^2}.
  \]
From the standard theory of elliptic equations
\beq
\|u\|_{H^2(\Omega)}\leq C \|\nabla \cdot ((\sigma_1 - \sigma_2)D{E}_2)\|_{L^2(\Omega)},
\eeq
which implies \eqref{diffeleH1}.
\end{proof}

\section{Uniqueness and stability}
With the notation of the previous section, we will show the well-posedness of the inverse problem in a certain functional space.

We prove the following theorem on the stability of the inverse problem.
\begin{theorem}
\label{stability}
Let $F(\sigma) =  \nabla\cdot (\sigma D E_{\sigma D}\times B_0)$. Suppose Assumption \ref{asmp:sigma} is satisfied and $\sigma_1-\sigma_2\in W^{1,\infty}_0$. If there exist constants $K$, $L$ and $\eta$ such that
\beq
\label{gragsigma}
\|\nabla \sigma_i\|_{L^{\infty}} \leq K
\eeq
\beq
\label{asu-lambda}
|1-\lambda| \leq \eta
\eeq
\beq
\label{asu-sigma12}
\|\nabla(\sigma_1-\sigma_2)\|_{L^2(\Omega)}\leq L \|\sigma_1-\sigma_2\|_{L^2(\Omega)},
\eeq
then
\beq
c\|\sigma_1-\sigma_2\|_{L^2(\Omega)} \leq  \|F(\sigma_1)-F(\sigma_2)\|_{L^2(\Omega)}
\eeq
holds for some positive constant $c$.
\end{theorem}
\begin{proof}
We denote $E_i=E_{\sigma_iD}$, $i=1,2$ and write the data difference as
\ben
F(\sigma_1)-F(\sigma_2) &=& \nabla\cdot (\sigma_1DE_1\times B_0) - \nabla\cdot (\sigma_2DE_2\times B_0) \nn\\
 &=& \nabla\cdot ((\sigma_1-\sigma_2)DE_1\times B_0)+\nabla\cdot (\sigma_2D(E_1-E_2)\times B_0).\nn
\een
 Then, we can rewrite $F(\sigma_1)-F(\sigma_2)$ as
\ben
F(\sigma_1)-F(\sigma_2) &=& I_1 +I_2 + I_3 +I_4,\nn
\een
where
\ben
I_1 &=& \nabla\cdot ((\sigma_1-\sigma_2)E_1\times B_0),\nn\\
I_2 &=& \nabla\cdot (\sigma_2(E_1-E_2)\times B_0),\nn\\
I_3 &=& \nabla\cdot ((\sigma_1-\sigma_2)(D-I)E_1\times B_0),\nn\\
I_4 &=& \nabla\cdot (\sigma_2(D-I)(E_1-E_2)\times B_0), \nn
\een
where $I$ is the identity matrix.

Next we multiply $F(\sigma_1)-F(\sigma_2)$ by $\sigma_1-\sigma_2$ and integrate over $\Omega$.
For $I_i$,  $i=1,2,3,4$, we  can estimate the integrals $\int_\Omega (\sigma_1 -\sigma_2) I_i$ separately. We have
\ben
\int_{\Omega}(\sigma_1-\sigma_2)I_1dx &=& \int_{\Omega} \bigg( (\sigma_1-\sigma_2) (\sigma_1-\sigma_2) \nabla\cdot (E_1\times B_0) \nn\\
& & + (\sigma_1-\sigma_2) \nabla(\sigma_1-\sigma_2)\cdot (E_1\times B_0) \bigg) dx \nn\\
& = & \frac{1}{2} \int_{\Omega}(\sigma_1-\sigma_2) (\sigma_1-\sigma_2) \nabla\cdot (E_1\times B_0) dx \nn\\
&=& \frac{1}{2}\|\sigma_1-\sigma_2\|^2_{L^2(\Omega)}.
\een
Here we use the equality $\nabla\cdot (E_1\times B_0)=1$ which can be easily checked from the identity $\nabla\cdot (E_1\times B_0)=B_0\cdot (\nabla \times E_1)- E_1\cdot(\nabla\times B_0)=1$. On the other hand,
\ben
\left|\int_{\Omega}(\sigma_1-\sigma_2)I_2dx\right| &=& \left|\int_{\Omega}(\sigma_1-\sigma_2) \nabla\sigma_2\cdot ((E_1-E_2)\times B_0) )dx\right| \nn\\
&\leq& KC\|\sigma_1-\sigma_2\|^2_{L^2(\Omega)}.
\een
Here the assumption (\ref{gragsigma}) and inequality (\ref{diffele}) have been used.

Now we turn to the terms $I_3$ and $I_4$. We have
\ben
\left|\int_{\Omega}(\sigma_1-\sigma_2)I_3dx\right| &=& \bigg| \int_{\Omega}(\sigma_1-\sigma_2) (\sigma_1-\sigma_2) \nabla\cdot ((D-I)E_1\times B_0) \nn\\
& & + (\sigma_1-\sigma_2) \nabla(\sigma_1-\sigma_2)\cdot ((D-I)E_1\times B_0) )dx \bigg| \nn\\
&=& \left|-\int_{\Omega}(\sigma_1-\sigma_2) \nabla(\sigma_1-\sigma_2) \cdot ((D-I)E_1\times B_0)dx\right|\nn\\
&\leq & \eta LC \|\sigma_1-\sigma_2\|^2_{L^2(\Omega)}.
\een
In the last inequality we have used estimate (\ref{inftyele}) together with the assumptions (\ref{asu-lambda}) and (\ref{asu-sigma12}). Finally, we have
\ben
\left|\int_{\Omega}(\sigma_1-\sigma_2)I_4dx\right| &=& \bigg|\int_{\Omega}(\sigma_1-\sigma_2) \sigma_2\nabla\cdot ((D-I)(E_1-E_2)\times B_0) \nn\\
& & + (\sigma_1-\sigma_2) \nabla\sigma_2\cdot ((D-I)(E_1-E_2)\times B_0) )dx \bigg| \nn\\
&=& \left|-\int_{\Omega}\sigma_2\nabla(\sigma_1-\sigma_2) \cdot ((D-I)(E_1-E_2)\times B_0)dx\right|\nn\\
&\leq & \eta LC \|\sigma_1-\sigma_2\|^2_{L^2(\Omega)}.
\een

Here we have used the assumptions (\ref{asu-lambda}), (\ref{asu-sigma12}) and inequality (\ref{diffele}).


\bigskip

Let $K$ and $\eta$ be such that $KC+2\eta LC <\frac{1}{2}$. We obtain
\ben
\int_{\Omega}(\sigma_1-\sigma_2)(F(\sigma_1)-F(\sigma_2)) \geq c \|\sigma_1-\sigma_2\|^2_{L^2(\Omega)},\nn
\een
for some constant $c$, which proves the theorem.
\end{proof}

Now we are ready to introduce a functional framework for which the inverse problem is well defined.
 We assume that $\sigma$ is known on the boundary of $\Omega$.  In what follows, we let $\sigma_*$, the true cross-property factor of $\Omega$, belong to a bounded convex subset of $C^{1,\beta}(\Omega)$ given by
$$ \widetilde{\mathcal{S}} = \{\sigma := \sigma_0 + \alpha | \,\alpha \in \mathcal{S}\},$$
where $\sigma_0$ is some positive function satisfying Assumption \ref{asmp:sigma} and
\begin{equation}\label{eq:mathcalS}
\begin{array}{ll}
\mathcal{S} =& \{\alpha\in  C_0^{1,\beta}(\Omega)|~  c_1\leq \alpha + \sigma_0 \leq c_2,~ |\nabla(\alpha + \sigma_0)|_{L^{\infty}}\leq K, ~ \\
&\quad\quad\quad\quad\quad\quad\quad\quad\| \nabla\alpha\|_{L^2(\Omega)} \le L\|\alpha\|_{L^2(\Omega)}, ~\|\alpha\|_{L^2(\Omega)}\le c_3 \}
\end{array}
\end{equation}
with $c_1, c_2, c_3$ and $c_3$ being positive constants. In other words, we can write
$
\widetilde{\mathcal{S}} = \sigma_0 + \mathcal{S}.$

It is clear that the distribution of the electric field ${E}_{\sigma D}$
depends nonlinearly on the factor $\sigma$ and $\nabla \cdot
(\sigma D{E}_{\sigma D}\times {B}_0)$ is nonlinear with respect
to $\sigma$.  We
first examine the Fr\'echet differentiability of the forward operator
$F$. Then, some useful properties of the Fr\'echet derivative at
$\sigma$, $DF[\sigma]$, are presented.

To introduce the Fr\'echet derivative, we consider the following Neumann boundary value problem
\begin{equation}\label{eqn:DF-2}
\left\{
  \begin{array}{lll}
   \nabla \cdot (\sigma D\nabla \varphi_h) & = - \nabla \cdot (h D{E}_{\sigma D}), &\qquad \mbox{ in } \Omega,
  \\
     (\sigma D\nabla \varphi_h + hD{E}_{\sigma D}) \cdot \nu & =  0 , &\qquad \mbox{ on } \partial\Omega,
  \end{array}
  \right.
\end{equation}
and
\begin{equation}\label{eqn:DFad-2}
\left\{
  \begin{array}{lll}
   \nabla \cdot (\sigma D\nabla \Phi_g) & = - \nabla \cdot (\sigma D (B_0\times\nabla g)), &\qquad \mbox{ in } \Omega,
  \\
     (\sigma D\nabla \Phi_g + \sigma D(B_0\times \nabla g)) \cdot \nu & =  0 , &\qquad \mbox{ on } \partial\Omega,
  \end{array}
  \right.
\end{equation}
where $h\in \mathcal{S}$ is the increment to the factor $\sigma$.

By the same arguments as those in \cite{Qiu2015}, together with Theorem \ref{stability}, it is natural to conclude the following result that insures the well-posedness of the inverse problem.
\begin{theorem}\label{thm:Fre-diff}
For $\sigma$ and $D$ satisfying Assumption~\ref{asmp:sigma} and (\ref{asu-lambda}), the operator $F$ is bounded and Fr\'echet differentiable at $\sigma \in \widetilde{\mathcal{S}}$. Its Fr\'echet derivative at $\sigma$, $DF[\sigma]$, is given by
  \begin{equation}\label{Fre-der}
    DF[\sigma](h) = \nabla \cdot ((\sigma D\nabla \varphi_h + h D{E}_{\sigma D}) \times {B}_0),
  \end{equation}
  where $\varphi_h$ solves \eqref{eqn:DF-2}.
Meanwhile, $DF[\sigma]^*$, i.e., the adjoint of $DF[\sigma]$ is defined as below,
\begin{equation}\label{Fread-der}
    DF[\sigma]^*(g) = -DE_{\sigma D}\cdot \nabla \Phi_g -\nabla g\cdot (DE_{\sigma D}\times B_0),
  \end{equation}
  where $\Phi_g$ solves (\ref{eqn:DFad-2}).
   Furthermore, we have the following stability result,
  \begin{equation}\label{DF-bd}
    c \|h\|_{L^2(\Omega)}\leq\|DF[\sigma](h)\|_{L^2} \leq C \|h\|_{L^2(\Omega)}, \quad \forall h\in \mathcal{S},
  \end{equation}
  for some constant $C$ which depends on $\lambda, c_1, c_2$ and $\Omega$ and the constant $c$ is defined in Theorem \ref{stability}.
\end{theorem}
\begin{proof}
The definition of $DF[\sigma](h)$ and (\ref{DF-bd}) follow from \cite{Qiu2015} and Theorem \ref{stability}. Here we only give a brief proof of the formulation of $DF[\sigma]^*$.

First by multiplying (\ref{eqn:DF-2}) by $\Phi_g$ and  (\ref{eqn:DFad-2}) by $\varphi_h$, we get after an integration by parts,
\beq\label{int-part}
\int_{\Omega}\sigma D(B_0\times \nabla g)\cdot \nabla \varphi_h dx = -\int_{\Omega} \sigma D\nabla \varphi_h\cdot \nabla \Phi_g dx = \int_{\Omega} hDE_{\sigma D}\cdot \nabla \Phi_g dx.
\eeq
Then we are ready to compute $DF[\sigma]^*(g)$. We have
\ben
\int_{\Omega} DF[\sigma](h) g dx &=& \int_{\Omega} \nabla \cdot ((\sigma D\nabla \varphi_h + h D{E}_{\sigma D}) \times {B}_0) g dx \nn \\
&=& -\int_{\Omega} ((\sigma D\nabla \varphi_h + h D{E}_{\sigma D}) \times {B}_0) \cdot\nabla g dx \nn \\
&=& -\int_{\Omega}  -\sigma D\nabla \varphi_h\cdot \nabla \Phi_g + h (D{E}_{\sigma D} \times {B}_0) \cdot\nabla g dx \nn \\
&=& -\int_{\Omega}  hDE_{\sigma D}\cdot \nabla \Phi_g + h (D{E}_{\sigma D} \times {B}_0) \cdot\nabla g dx \nn .
\een
This proves (\ref{Fread-der}).
\end{proof}

\section{The reconstruction method}
\subsection{Optimization scheme}
It is natural to formulate the reconstruction problem for $\sigma_*$ as a least-square problem. To find  $\sigma_*$ we minimize the functional
\begin{equation*}
    J(\sigma) = \frac{1}{2}\|F(\sigma) - F(\sigma_*)\|^2_{L^2 (\Omega)}
\end{equation*}
over $\sigma \in \widetilde{\mathcal{S}}$.

We can now apply the gradient descent method to minimize the discrepancy functional $J$. Define the iterates
\begin{equation} \label{defeta}
\sigma_{n+1}= T[\sigma_n] - \mu D
J[T[\sigma_n]],
\end{equation}
where $\mu >0$ is the step size and
$T[f]$ is any approximation of the Hilbert projection from $L^2(\Omega)$ onto $\overline{\widetilde{\mathcal{S}}}$ with $\overline{\widetilde{\mathcal{S}}}$ being the closure of $\widetilde{\mathcal{S}}$ (in the $L^2$-norm).

The presence of the projection $T$ is necessary because $
\sigma_n$ might not be in $\widetilde{\mathcal{S}}$.

Using the definition of $J$ we can show that the optimal control algorithm (\ref{defeta}) is nothing else than the following projected Landweber iteration \cite{landweber, Hankeetal:nm1995, Hoop2015} given by
\beq\label{minimizing sequence}
\begin{array}{rcl}
\sigma_{n+1}  = T[\sigma_n] - \mu \D
F^*[T[\sigma_n]] (F(T[\sigma_n])-F(\sigma_*)) .
\end{array}
\eeq

For completeness, we state the convergence result of Landweber scheme here without proof. We refer to \cite{Ammari14, Hoop2015} for  details.

\begin{theorem}
	The sequence defined in \eqref{minimizing sequence} converges to the true cross-property factor $\sigma_*$ of $\Omega$ in the following sense: there exists $\epsilon>0$ such that if
	$\| T[\sigma_1] -  \sigma_*\|_{L^2(\Omega)} < \epsilon$, then
	$$ \lim_{n \rightarrow + \infty} \| \sigma_n - \sigma_*\|_{L^2(\Omega)} =0.$$
\end{theorem}


%
%

\subsection{A quasi-Newton method}
It has been observed in \cite{Qiu2015} that the challenge of the Landweber iteration lies in the difficulty of evaluating the adjoint operator of the Fr\'echet derivative.
To avoid taking too many derivatives, we introduce a more efficient way to reconstruct the conductivity. This is a generalization of the quasi-Newton method proposed in \cite{Qiu2015} for the anisotropic case with known conformal class.
In the following, we describe this algorithm and prove its convergence in $\widetilde{\mathcal{S}}$.

Let $\sigma$ be the scalar conductivity distribution function and let $D$ be the known conformal class matrix-valued function. The forward operator is given by
\[
F(\sigma) = \nabla \cdot(\sigma DE_{\sigma D}\times B_0),
\]
where $E_\sigma$ satisfies the system
\[
\left\{
\begin{aligned}
  &\nabla \cdot(\sigma DE_{\sigma D}) && = 0,  & \mbox{ in } \Omega,
  \\
  &\nabla \times E_{\sigma D} && =  B_1,   & \mbox{ in } \Omega,
  \\
  & \sigma DE_{\sigma D} \cdot \nu && = 0,   &\mbox{ on } \partial\Omega.
\end{aligned}
\right.
\]

\begin{alg}\label{algorithm}~\\

\textbf{ Step 0.}  Select an initial conductivity $\sigma_1\in \widetilde{\mathcal{S}}$ and set $k=1$.

  \textbf{ Step 1.} Calculate the associated electric field $E_k$ by solving the
        boundary value problem
      \begin{equation}\label{eqn:update_E}
      \left\{
      \begin{array}{lll}
        \nabla \times E_k & = B_1, &\qquad \mbox{ in }
        \Omega,
          \\
        \nabla \cdot (\sigma_k DE_k) & =0, &\qquad \mbox{ in } \Omega,
          \\
          \sigma_k DE_k \cdot \nu &= 0 , &\qquad \mbox{ on }
            \partial\Omega .
       \end{array}
       \right.
      \end{equation}

 \textbf{ Step 2.}  Calculate the updated conductivity by solving the stationary advection-diffusion equation with the inflow boundary condition:
      \begin{equation}\label{eqn:update_sigma}
      \left\{
      \begin{array}{rll}
        \nabla \cdot (\sigma_{k+1/2} DE_{k}\times B_0) & = g , &\qquad \mbox{ in } \Omega,
          \\
        \sigma_{k+1/2} & = \sigma_* , & \qquad \mbox{ on } \partial\Omega^-,
        \end{array}
       \right.
      \end{equation}
      where $\partial\Omega^- = \{ x\in \partial \Omega \mid DE_k(x) \times B_0 \cdot \nu(x) < 0 \}$.

\textbf{ Step 3.}
Let $\sigma_{k+1}=T[\sigma_{k+1/2}]$, where $T$ is the Hilbert projection operator onto $\tilde{S}$.
 Set $k=k+1$ and go to (\ref{eqn:update_E}).
\end{alg}

\subsection{Convergence analysis}
In the algorithm above, one updates the electric field $E$ and then updates the cross-factor $\sigma$ later. Using the same argument as for proving the well-posedness, we could get the following convergence results.
\begin{theorem}\label{conv-alg}
Suppose that the cross-factor $\sigma_*\in \widetilde{\mathcal{S}}$ and $D$ satisfies Assumption \ref{asmp:sigma} and \eqref{asu-lambda}. Let
$\{\sigma_k\}$ be determined by the Algorithm \ref{algorithm}. Then for proper constants $K, L$ and $\eta$ in \eqref{eq:mathcalS} and \eqref{asu-lambda}, there exists a constant $c<1$ such that
\beq
\|\sigma_{k+1}-\sigma_*\|_{L^2(\Omega)}\leq c\|\sigma_{k}-\sigma_*\|_{L^2(\Omega)}.
\eeq

\end{theorem}
\begin{proof}
Note that $T$ is a projection and $\widetilde{\mathcal{S}}$ is convex. Then, we have that $T$ is nonexpansive and $\|\sigma_{k+1}-\sigma_*\|\leq \|\sigma_{k+1/2}-\sigma_*\|$ follows.
It is left to estimate $\|\sigma_{k+1/2}-\sigma_*\|$.

First we subtract $\nabla\cdot(\sigma_*DE_k\times B_0)$ from both sides of \eqref{eqn:update_sigma}to get
\beq
\nabla\cdot((\sigma_{k+1/2}-\sigma_*)DE_k\times B_0)=\nabla\cdot(\sigma_*D(E_*-E_k)\times B_0).
\eeq
Multiplying by $\sigma_{k+1/2}-\sigma_*$ and integrating over $\Omega$ yields
\beq
\int_{\Omega}(\sigma_{k+1/2}-\sigma_*)\nabla\cdot((\sigma_{k+1/2}
-\sigma_*)DE_k\times B_0)=\int_{\Omega}(\sigma_{k+1/2}-\sigma_*)\nabla\cdot(\sigma_*D(E_*-E_k)\times B_0).
\eeq
We split the terms into
\ben
&&\int_{\Omega}(\sigma_{k+1/2}-\sigma_*)\nabla\cdot((\sigma_{k+1/2}
-\sigma_*)E_k\times B_0) + (\sigma_{k+1/2}-\sigma_*)\nabla\cdot((\sigma_{k+1/2}
-\sigma_*)(D-I)E_k\times B_0)\nn\\
&=&\int_{\Omega}(\sigma_{k+1/2}-\sigma_*)\nabla\sigma_*\cdot(D(E_*-E_k)\times B_0) + (\sigma_{k+1/2}-\sigma_*)\sigma_*\nabla\cdot((D-I)(E_*-E_k)\times B_0). \nn
\een

Integrating by parts gives $\int_{\Omega}(\sigma_{k+1/2}-\sigma_*)\nabla\cdot((\sigma_{k+1/2}
-\sigma_*)E_k\times B_0) = \frac{1}{2}\|\sigma_{k+1/2}-\sigma_*\|_{L^2(\Omega)}^2$.

The second term in the left hand side can be estimated as follows:
\ben
&&\left|\int_{\Omega} (\sigma_{k+1/2}-\sigma_*)\nabla\cdot((\sigma_{k+1/2}
-\sigma_*)(D-I)E_k\times B_0)\right|\nn\\
 &=& \left|\int_{\Omega} (\sigma_{k+1/2}-\sigma_*)^2\nabla\cdot((D-I)E_k\times B_0)\right|\nn\\
 &\leq &C\eta \|\sigma_{k+1/2}-\sigma_*\|_{L^2(\Omega)}^2.\nn
\een
Here  the smallness of $D-I$ and the $C^1$ property of $E_k$ \eqref{C1} have been used.

For the right hand side, we have
\beq
\left|\int_{\Omega}(\sigma_{k+1/2}-\sigma_*)\nabla\sigma_*\cdot(D(E_*-E_k)\times B_0)\right|\leq CK\|\sigma_{k+1/2}-\sigma_*\|_{L^2(\Omega)}\|\sigma_{k}-\sigma_*\|_{L^2(\Omega)},\nn
\eeq
and
\ben
\left|\int_{\Omega}(\sigma_{k+1/2}-\sigma_*)\sigma_*\nabla\cdot((D-I)(E_*-E_k)\times B_0)\right| & \leq & C\eta \|\sigma_{k+1/2}-\sigma_*\|_{L^2(\Omega)}\|\sigma_{k}-\sigma_*\|_{H^1(\Omega)}\nn\\
&\leq & C(L+1)\eta \|\sigma_{k+1/2}-\sigma_*\|_{L^2(\Omega)}\|\sigma_{k}-\sigma_*\|_{L^2(\Omega)}.\nn
\een
Here we have used property \eqref{diffeleH1} and the fact that $\sigma_k \in \tilde{S}$.

With the above estimates, as we did in Theorem \ref{stability}, let $KC+\eta (L+2)C <\frac{1}{2}$.  We derive
\beq
\|\sigma_{k+1/2}-\sigma_*\|_{L^2(\Omega)}\leq c\|\sigma_{k}-\sigma_*\|_{L^2(\Omega)},
\eeq
 where $c$ is a constant smaller than $1$. Hence,
\beq
\|\sigma_{k+1}-\sigma_*\|_{L^2(\Omega)}\leq c\|\sigma_{k}-\sigma_*\|_{L^2(\Omega)},
\eeq
which proves the theorem.
\end{proof}

\subsection{Numerical experiments}
In this section, we present some numerical experiments to validate the reconstruction method proposed in Algorithm~\ref{algorithm} and evaluate its robustness to measurement noise.
To simplify the computation, we convert this three-dimensional problem into an equivalent two-dimensional problem assuming that the domain of interest is the cube $[0,1]^3$ and the conductivity and the diffusion tensors are invariant along the third dimension.  Moreover, we assume that
the diffusion tensor $D$ is of form
\beq
\label{diff-mat}
D=\begin{pmatrix}
d_{11} & d_{12} &  0 \\
d_{21} & d_{22} &  0 \\
 0     &    0   &  1
\end{pmatrix},
\eeq
where $d_{ij}$'s are constant plus some perturbations as shown in Figure~\ref{fig:diff-matrix}. The non-zero part of the perturbation functions are used to characterize the anisotropy.

We use a uniform finite element triangular mesh over the two-dimensional unit square. The number of cells is $256$ in each direction. The total number of triangles and vertices are $2^{17}$ and $257^2$, respectively. Both the elliptic equation with a Neumann boundary condition and the stationary advection-diffusion equation are solved using the finite element method of first order implemented with FEniCS \cite{Logg2012}. The internal data $F(\sigma)$ used for the reconstruction are synthetic data that are generated using the same solver.
These data are commonly used to refer to the ``noise-free'' data, although they may contain some numerical errors.

For all examples, we use the same initial guess, constant function $0.2$, and the same true cross-property factor (Figure~\ref{fig:sigma}) given by
\[
\sigma(x_1,x_2) =
\left\{
\begin{aligned}
  0.6 & , \quad r \leq 0.12,
  \\
  0.4 s^3(6s^2-15s+10) + 0.2 & , \quad 0.12 < r <0.46,
  \\
  0.2 & , \quad \mbox{others },
\end{aligned}
\right.
\]
where $r(x_1,x_2) = \sqrt{(x_1-0.5)^2 + (x_2-0.5)^2}$ and $s= \frac{0.46 - r}{0.12}$.
The internal data generated with the diffusion tensor as in \eqref{diff-mat} is shown in Figure~\ref{fig:data}. We also produce the data in the isotropic case (Figure~\ref{fig:idt-data}). The effect of the anisotropy can be observed clearly. The error-decay of the reconstruction with the noise-free data is shown in Figure~\ref{fig:error-behav}. The final error is smaller than $2\times 10^{-3}$. We only display the last iterate here.

This inverse problem bears a Lipschitz type stability and we avoid lowering the regularity of the cross-property factor using Algorithm~\ref{algorithm}. Therefore, the robustness of the reconstruction scheme to noisy data is expected. We perform the numerical tests with noisy data by perturbing the internal functional $g$ in the following way:
\[
g_\delta = g  + \delta \|g\|\frac{w}{\|w\|},
\]
where $w$ is a function taking values uniformly distributed in $[-1, 1]$ and $\delta$ is the noise level.

\begin{figure}
\centering
\subfloat[$d_{11}$]{\includegraphics[width=.33\textwidth,bb=0 0 480 360]{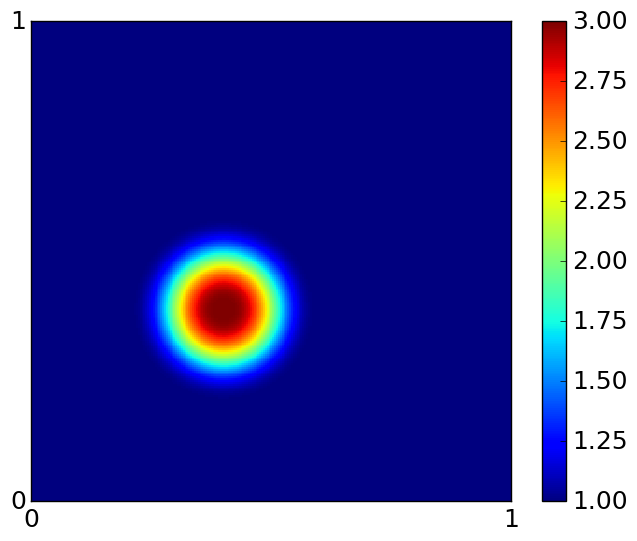}}
\subfloat[$d_{12}$]{\includegraphics[width=.33\textwidth,bb=0 0 480 360]{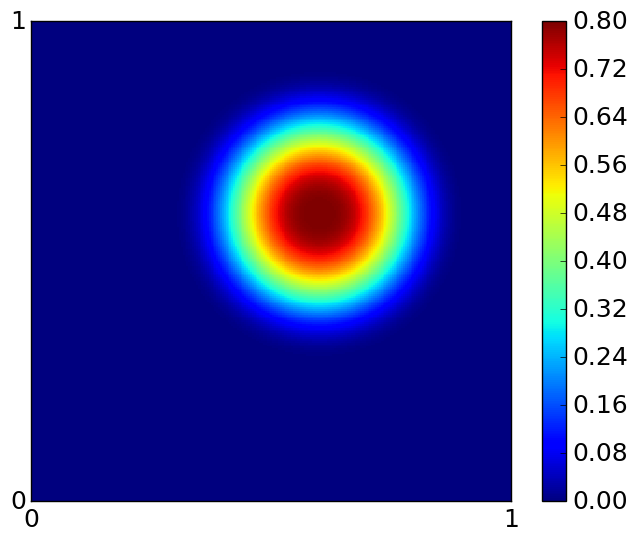}}
\subfloat[$d_{22}$]{\includegraphics[width=.33\textwidth,bb=0 0 480 360]{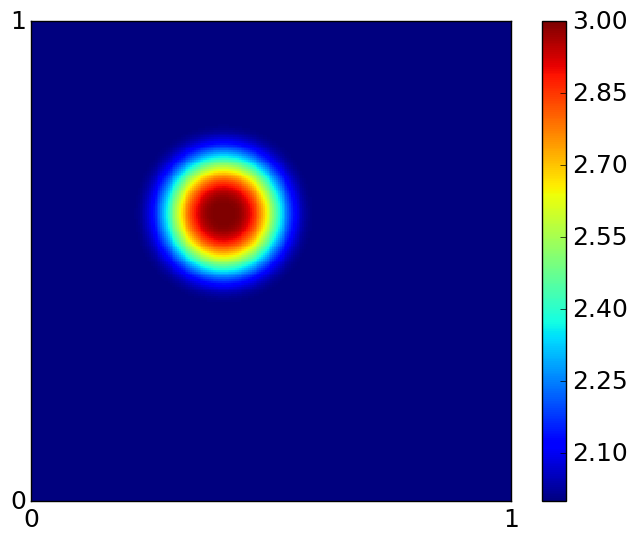}}
\caption{Components of the diffusion tensor.}
\label{fig:diff-matrix}
\end{figure}

\begin{figure}
\centering
\subfloat[True conductivity]{\includegraphics[width=.33\textwidth,bb=0 0 480 360]{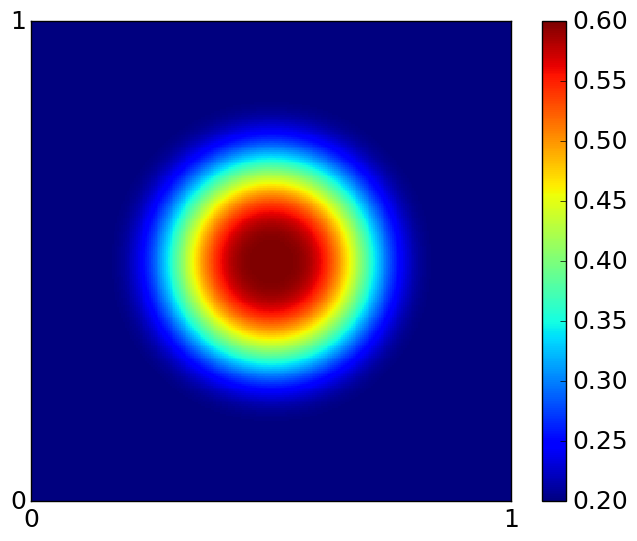} \label{fig:sigma}}
\subfloat[Internal data]{\includegraphics[width=.33\textwidth,bb=0 0 480 360]{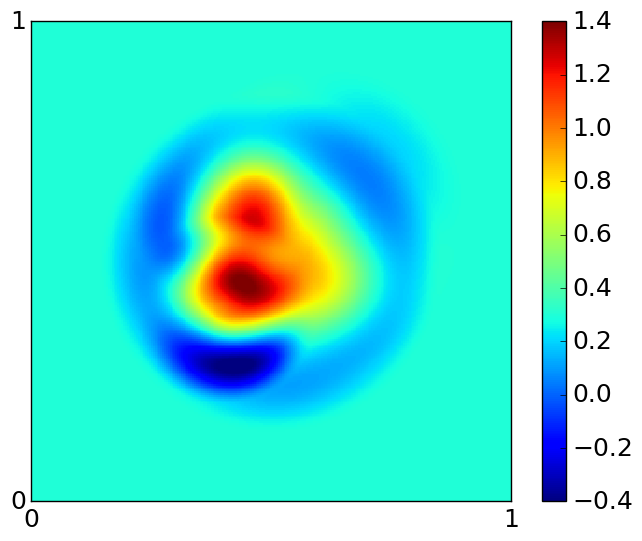}\label{fig:data}}
\subfloat[Internal data with isotropic diffusion tensor]{\includegraphics[width=.33\textwidth,bb=0 0 480 360]{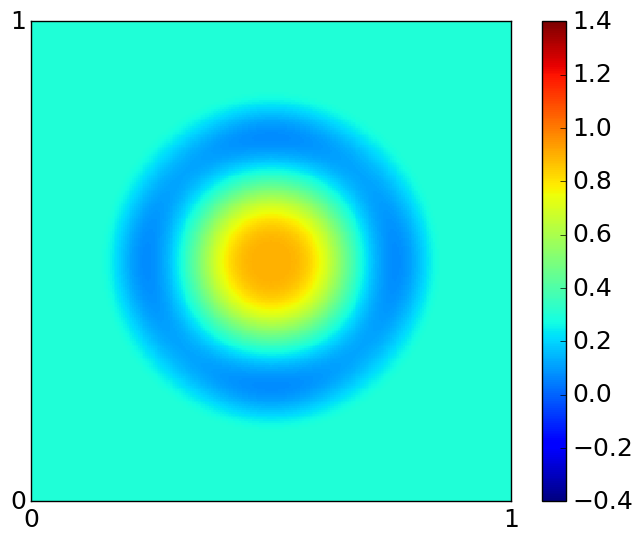}\label{fig:idt-data}}
\caption{Conductivity distribution and the internal data.}
\label{fig:sig-data}
\end{figure}

Figure~\ref{fig:noisy-recon-rts} shows the noisy data with noise level $\delta = 24\%$ and the reconstructed cross-property factor. We do not use further regularization techniques since the regularization method may depend on the type of the noise in practical cases. But the projection onto the feasible space acts as a regularization scheme.

\begin{figure}
\centering
\subfloat[Noisy data]{\includegraphics[width=.33\textwidth,bb=0 0 480 360]{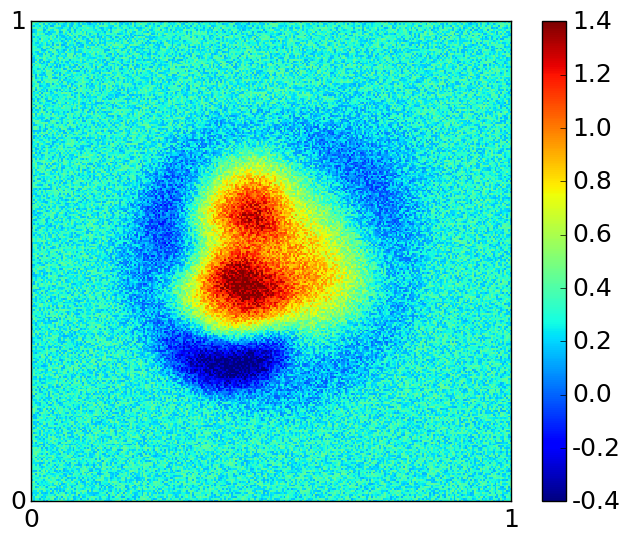} }
\subfloat[Reconstructed $\sigma$]{\includegraphics[width=.33\textwidth,bb=0 0 480 360]{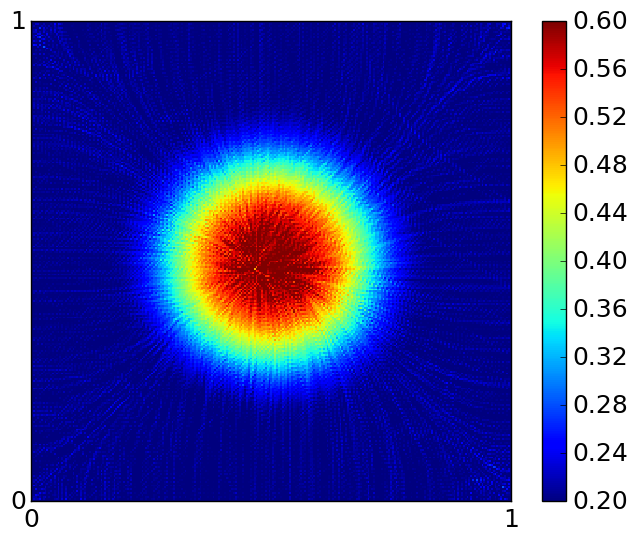}}
\caption{Reconstruction with noisy data ($\delta = 24\%$).}
\label{fig:noisy-recon-rts}
\end{figure}

\begin{figure}
\centering
\subfloat[Error decay with noise-free data.]{
\includegraphics[width=.45\textwidth,bb=0 0 500 450]{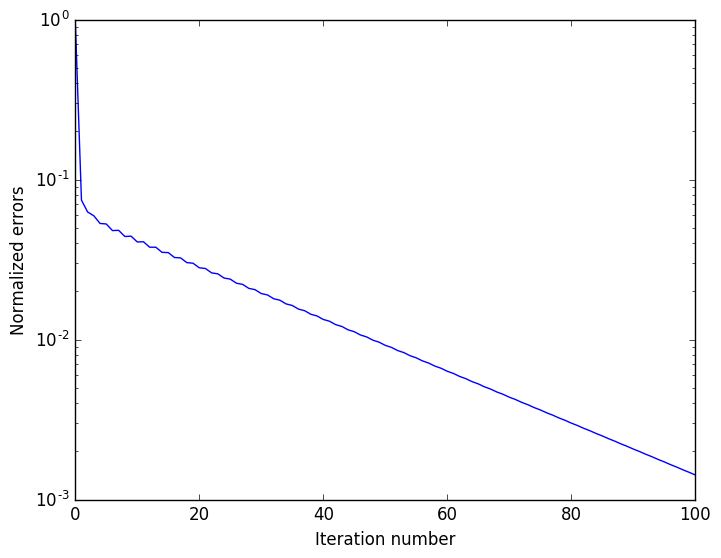}
\label{fig:error-behav}
}
\subfloat[Reconstruction error versus noise level.]{
\includegraphics[width=.45\textwidth,bb=0 0 500 450]{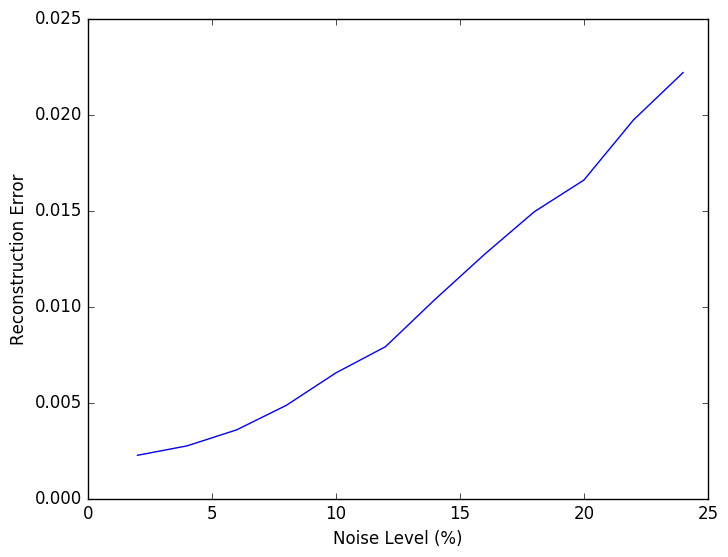}
\label{fig:recon-noise-err}
}
\caption{Reconstruction error.}
\label{fig:recon-rts}
\end{figure}


\section{Concluding remarks}

In this paper, we have considered the reconstruction of an anisotropic
conductivity from MAT-MI data which is conformal to a known diffusion tensor measured from
DTI. The data is the internal functional $\nabla\cdot (\sigma
E_{\sigma}\times B_0)$ throughout the domain.
We have analyzed the linearization of the problem and the stability of the
inversion. A local Lipschitz type stability estimate has been established for a
certain class of anisotropic conductivities. A quasi-Newton type reconstruction method with
projection has been introduced and its convergence has been proved. Numerical
experiments demonstrate the effectiveness of the proposed approach and its
robustness to noise.

In light of the numerical experiments, we have the following observations.
\begin{enumerate}
  \item The effect of the electrical anisotropy is remarkably significant
and can not be neglected in the reconstruction of electrical conductivity in MAT-MI.
  \item There is still a room for improvement of the admissible class of
conductivities. The convergence of the proposed algorithm has been observed
for more general cases.
  \item For the inversion with noisy data, oscillation in the reconstructed
conductivity is observed. Regularization methods prompting sparsity, such
as total variation regularization may be employed for a more
stable reconstruction.
\end{enumerate}

The present work leads to many important unanswered questions.  First, is it possible to determine an anisotropic conductivity $\gamma (x)$ from MAT-MI data alone?  Recall that each static field $B_0$ leads to a new measurement.  It would be important to know how many static fields and their associated measurements are needed to determine $\gamma(x)$ uniquely. Second, how can the present approach be modified to address this very question? It is clear that the current analysis relies on vector identities which do not easily generalize to handle anisotropic conductivity.  Finally, if we can answer these questions, then how do we devise stable and accurate computational methods to image anisotropic conductivity? We therefore must accept the conclusion that the present work is the first step towards a fully MAT-MI method of imaging anisotropic conductivity.

%
\small
\def\cprime{$'$}
  \providecommand{\noopsort}[1]{}\def\ocirc#1{\ifmmode\setbox0=\hbox{$#1$}\dimen0=\ht0
  \advance\dimen0 by1pt\rlap{\hbox to\wd0{\hss\raise\dimen0
  \hbox{\hskip.2em$\scriptscriptstyle\circ$}\hss}}#1\else {\accent"17 #1}\fi}

\end{document}